\documentclass{amsart}

\usepackage{amssymb,latexsym,amscd}

\numberwithin{equation}{section}

\newtheorem{theorem}{Theorem}[section]

\newtheorem{proposition}[theorem]{Proposition}

\newtheorem{corollary}[theorem]{Corollary}

\newtheorem{lemma}[theorem]{Lemma}

\theoremstyle{definition}

\newtheorem{remark}[theorem]{Remark}

\def\proof{\smallskip\noindent {\bf Proof.\ }}
\def\endproof{\hfill$\square$\medskip}

\def\kk{\Bbbk}

\begin{document}


\title
{Generalized adjoint actions}

\author{Arkady Berenstein}
\address{\noindent Department of Mathematics, University of Oregon,
Eugene, OR 97403, USA} \email{arkadiy@math.uoregon.edu}

\author{Vladimir Retakh}
\address{Department of Mathematics, Rutgers University, Piscataway, NJ 08854, USA}
\email{vretakh@math.rutgers.edu}

\thanks{This work was partially supported by the NSF grant~DMS-1403527 (A.~B.) 
and by the NSA grant H98230-14-1-0148~(V.~R.)}

\begin{abstract} The aim of this paper is to generalize the classical formula 
$\displaystyle {e^xye^{-x}=\sum_{k\ge 0} \frac{1}{k!} (ad~x)^k(y)}$ by replacing $e^x$ with any formal power series
$\displaystyle {f(x)=1+\sum_{k\ge 1} a_kx^k}$. We also obtain combinatorial 
applications to $q$-exponentials, $q$-binomials, and Hall-Littlewood polynomials.

\end{abstract}

\maketitle

\section{Notation and main results}
One of the most fundamental tools in Lie theory, the adjoint action of Lie groups on their Lie algebras, 
is based on the following formula:

\begin{equation}
\label{eq:conj classical}
e^xye^{-x}=e^{ad~x}(y)=\sum_{k\ge 0} \frac{1}{k!} (ad~x)^k(y)\ ,
\end{equation}
where $(ad~x)^k(y)=[x,[x,\ldots,[x,y],\ldots]]$ and $[a,b]=ab-ba$. 

The aim of this paper is to generalize \eqref{eq:conj classical} by replacing $e^t$ with any formal power series
\begin{equation}
\label{eq:power series}
f=f(t)=1+\sum_{k\ge 1} a_kt^k
\end{equation} 
over a field $\kk$.

For any formal power series \eqref{eq:power series} over $\kk$ define polynomials
$$P_k(t)=P_{f,k}(t)=(-1)^k\det\begin{pmatrix}
1&a_1t&a_2t^2&\ldots &a_kt^k\\
1&a_1&a_2&\ldots &a_k\\
0&1&a_1&\ldots &a_{k-1}\\
\vdots&\vdots&\ddots&\ddots&\ddots\\
0&0&\ldots &1&a_1 \end{pmatrix}$$
for $k=0,1,2,\ldots$. 
(with the convention that $P_0(t)=1$). Clearly, $P_k(1)=0$ for $k\ge 1$. Using Cramer's rule with respect to the last column, one obtains a recursion $P_k(t)=a_kt^k - \sum\limits_{i=1}^k a_iP_{k-i}(t)$.

The following result is, probably,  well-known (for readers' convenience, we prove it in Section \ref{sect:proofs}).
\begin{theorem}
\label{th:quotient of series} For any power series $f(t)$ as in \eqref{eq:power series}, one has
\begin{equation}
\label{eq:ftxfx}
\displaystyle{\frac{f(tx)}{f(x)}=\sum\limits_{k\ge 0}  P_{f,k}(t)\cdot x^k} 
\end{equation}
and
\begin{equation}
\label{eq:Pfk}
P_{f,k}(st)=\sum\limits_{i=0}^k P_{f,i}(s)P_{f,k-i}(t)t^i
\end{equation}
for all $k\ge 0$.
\end{theorem}

Furthermore, for any algebra ${\mathcal A}$ over $\kk$, a subset ${\bf q}=\{q_1,\ldots,q_k\}\subset \kk$,  $x,y\in {\mathcal A}$, and 
$k\ge 1$ define 
$$(ad~ x)^{\bf q}(y)=[x,[x,\ldots,[x,y]_{q_1},\ldots]_{q_{k-1}}]_{q_k}$$
where $[a,b]_q:=ab-qba$.
It is easy to see that 
\begin{equation}
\label{eq:adq}
(ad~ x)^{\bf q}(y)=\sum_{j=0}^k (-1)^je_j(q_1,\ldots,q_k)\cdot x^{k-j}y x^j\ ,
\end{equation}
where $e_j(q_1,\ldots,q_k)$ is the $j$-th elementary symmetric function.

\begin{theorem} 
\label{th:general conjugation}
Let ${\mathcal A}$ be $\kk$-algebra  and suppose that $f$ is any power series \eqref{eq:power series} 
with $a_k\ne 0$ for $k\ge 1$. Then
\begin{equation}
\label{eq:conj general}
f(x)yf(x)^{-1}=y+\sum_{k\ge 1} a_k (ad~x)^{{\bf q}_k}(y)\ ,
\end{equation}
for any $x,y\in {\mathcal A}$, where ${\bf q}_k=\{q_{1k},\ldots,q_{kk}\}$ is the set of roots of $P_{f,k}(t)$.

\end{theorem}

\begin{remark} A formula for $f(x)yf(x)^{-1}$ 
without assumption that all $a_k\ne 0$ is given in Proposition \ref{pr:conjugation general}.

\end{remark}

\begin{remark} Strictly speaking, the formula \eqref{eq:conj general}, similarly to \eqref{eq:conj classical} requires a completion of ${\mathcal A}$. One can bypass this by replacing $x$ with $\tau\cdot x$ where $\tau$ is a purely transcendental element of $\kk$ so that the right hand side of \eqref{eq:conj general} becomes a power series in $\tau$ (and, maybe extending $\kk$ if it lack such an element).

\end{remark}

\begin{remark} The subsets ${\bf q}_k$ may belong to an extension of $\kk$, however, the operators $(ad~x)^{{\bf q}_k}$ are defined over $\kk$ due to \eqref{eq:adq} because all symmetric functions in ${\bf q}_k$ belong to $\kk$. 

\end{remark}

It is easy to see that if $a_k=\frac{1}{k!}$ for all $k$, then $P_k(t)=\frac{(t-1)^k}{k!}$ which immediately recovers \eqref{eq:conj classical}. 
Suppose now that $a_k=\frac{1}{[k]_q!}$ for all $k$, where $k_q!=[1]_q\cdots [k]_q$ is the $q$-factorial and 
$[\ell]_q=1+q+\cdots +q^{\ell-1}$.
We will show (Proposition \ref{pr:qbinom factorization}) that $P_{f,k}(t)=\displaystyle{\frac{(t-1)(t-q)\cdots (t-q^{k-1})}{[k]_q!}}$ for 
$\displaystyle{f(t)=e_q^t=\sum\limits_{k\ge 0} \frac{t^k}{[k]_q!}}$, 
therefore, recover the following famous result (see e.g., \cite{V}).

\begin{theorem} Let $\displaystyle{e_q^x=\sum\limits_{k\ge 0} \frac{x^k}{[k]_q!}}$ be the $q$-exponential. Then
$\displaystyle{e_q^x\cdot y\cdot (e_q^x)^{-1}=\sum_{k\ge 0} \frac{1}{[k]_q!} (ad~x)^{\{1,q,\ldots,q^{k-1}\}}(y)}
$.
\end{theorem}

On the other hand, combining Theorem \ref{th:quotient of series} and Proposition \ref{pr:qbinom factorization}, we recover the following well-known properties of $q$-exponentials and $q$-binomials:

$$ e_q^{q^nx}=e_q^x\left(1+\sum_{k=1}^n \frac{(q^n-1)(q^n-q)\cdots (q^n-q^{k-1})}{[k]_q!}x^k\right)$$
for $n\ge 0$, 
in particular,
$$e_q^{qx}=e_q^x\cdot (1+(q-1)x)$$
and 
$$1+\sum_{k=1}^n \frac{(q^n-1)(q^n-q)\cdots (q^n-q^{k-1})}{[k]_q!}x^k=\prod_{i=1}^n  (1+(q-1)q^{i-1}x) \ .$$

We conclude with a curious observation that the polynomials $P_{f,k}(t)$ are related to the Hall-Littlewood symmetric polynomials. 

\begin{proposition} 
\label{pr:Hall-Littlewood}
Suppose that $\displaystyle{f(t)=\prod\limits_{k\ge 1} (1-x_kt)}$. Then 
$$P_{f,k}(t)=Q_{(k)}({\bf x};t)$$ 
for all $k\ge 0$, where  ${\bf x}=\{x_k,k\ge 0\}$ is viewed as an infinite set of variables,  $Q_\lambda({\bf x};t)$ is Hall-Littlewood polynomial 
(\cite[Section 3.2]{M}), and $(k)$ is a one-row Young diagram with $k$ cells. In particular, 
$$\displaystyle{Q_{(k)}({\bf x};t)=(-1)^k\det\begin{pmatrix}
1&-e_1t&e_2t^2&\ldots &(-1)^ke_kt^k\\
1&-e_1&e_2&\ldots &(-1)^ke_k\\
0&1&-e_1&\ldots &(-1)^{k-1}e_{k-1}\\
\vdots&\vdots&\ddots&\ddots&\ddots\\
0&0&\ldots &1&-e_1 \end{pmatrix}}$$
for all $k\ge 0$, where $e_k=e_k({\bf x})$ is the $k$-th elementary symmetric function. 
\end{proposition}

\noindent {\bf Acknowledgments}. We gratefully acknowledge the support of Centre de Recerca Matem\`{a}tica, Barcelona, where this work was accomplished. We are grateful to Eric Rains for pointing out the connection with Hall-Littlewood polynomials.

%
%
%

\section{Proofs}

\label{sect:proofs}
\noindent {\bf Proof of Theorem \ref{th:quotient of series}}. 
We need the following well-known fact (attributed to Wronski, see e.g., \cite{I}).

\begin{lemma} 
\label{le:inverse power series}
Let $f$ be any formal power series \eqref{eq:power series}. Then  $\displaystyle{\frac{1}{f(t)}=1+\sum_{k\ge 1} D_k(f)t^k}$,
where  
$$
D_k(f)=(-1)^k\det\begin{pmatrix} 
a_1&a_2&a_3&\ldots &a_k\\
1&a_1&a_2&\ldots &a_{k-1}\\
0&1&a_1&\ldots &a_{k-2}\\
\vdots&\vdots&\ddots&\ddots&\ddots\\
0&0&\ldots &1&a_1 \end{pmatrix}
$$
(with the convention $D_0(f)=1$).

\end{lemma}

The following generalization of Lemma \ref{le:inverse power series} is, apparently, well-known (for readers' convenience we prove it here).

\begin{lemma}
\label{le:quotient power series} Let $f(t)=1+\sum\limits_{k\ge 1} a_kt^k$, $g(t)=1+\sum\limits_{k\ge 1} b_kt^k$ be formal power series. Then
$$\displaystyle{\frac{g(t)}{f(t)}=\sum_{k\ge 0} D_k(g,f)t^k}\ ,$$ where $D_k(g,f)=(-1)^k\det\begin{pmatrix}
1&b_1&\ldots&b_{k-1}&b_k\\
1&a_1&\ldots&a_{k-1}&a_k\\
0&1&\ldots&a_{k-2}&a_{k-1}\\
\vdots&\vdots&\ddots&\ddots&\ddots\\
0&0&\ldots&1&a_1 \end{pmatrix}
$ (with the convention $D_0(g,f)=1$).
\end{lemma}

\begin{proof} Indeed, using Lemma \ref{le:inverse power series}, we obtain (with the convention $b_0=1$):
$$\frac{g(t)}{f(t)}=g(t)\cdot \frac{1}{f(t)}=\left(\sum_{i\ge 0} b_it^i\right)\left( \sum_{j\ge 0} D_j(f)t^j\right)=\sum_{k\ge 0} d_k t^k$$
where  
$$d_k=\sum_{i=0}^k b_i D_{k-i}(f)=(-1)^k\det\begin{pmatrix}
b_0 &b_1&\ldots&b_{k-1}&b_k\\
1&a_1&\ldots&a_{k-1}&a_k\\
0&1&\ldots&a_{k-2}&a_{k-1}\\
\vdots&\vdots&\ddots&\ddots&\ddots\\
0&0&\ldots&1&a_1 \end{pmatrix}
$$
by Cramer rule because $b_iD_{k-i} (f)=(-1)^k \det\begin{pmatrix}
0 &0&\ldots&b_i& \ldots &0\\
1&a_1&\ldots&a_i& \ldots&a_k\\
0&1&\ldots&a_{i+1}&\ldots & a_{k-1}\\
\vdots&\vdots&\ddots &\ddots&\ddots&\ddots\\
0&0&\ldots& 1 &a_1&a_2 \\
0&0&\ldots& \ldots &1&a_1 
\end{pmatrix}
$.

The lemma is proved.
\end{proof}

Then taking $b_k=a_kt^k$ for $k\ge 1$ in Lemma \ref{le:quotient power series}, 
we obtain \eqref{eq:ftxfx}. 

To prove \eqref{eq:Pfk}, compute $\frac{f(stx)}{f(x)}$ in two ways, using the first assertion:
$$\displaystyle{\frac{f(stx)}{f(x)}=\sum\limits_{k\ge 0}  P_{f,k}(st)\cdot  x^k}$$
and 
$\displaystyle{\frac{f(stx)}{f(x)}=
\frac{f(stx)}{f(tx)}\cdot \frac{f(tx)}{f(x)}=\left(\sum\limits_{i\ge 0}  P_{f,i}(s)\cdot (tx)^i\right)
\left(\sum\limits_{j\ge 0}  P_{f,j}(t) \cdot x^j\right)}$.
Comparing the coefficients of $x^k$ in both series, we obtain \eqref{eq:Pfk}.

Theorem \ref{th:quotient of series} is proved.
\endproof

\noindent {\bf Proof of Theorem \ref{th:general conjugation}}.
We need the following result.

\begin{proposition} 
\label{pr:conjugation general}

For any power series $f$ as in \eqref{eq:power series} one has:

(a) $P_{f,k}(t)=\sum\limits_{j=0}^k a_{k-j}D_j(f)\cdot t^j$ for all $k\ge 0$.

(b) $f(x)y f(x)^{-1}=y+z_1+z_2+\cdots$, 
where 
$
z_k=\sum\limits_{i=0}^k a_iD_{k-i}(f)\cdot x^iyx^{k-i}
$
for all $k\ge 1$.
\end{proposition}

\begin{proof} Prove (a). Indeed, using Lemma \ref{le:inverse power series}, we obtain:
$$\frac{f(tx)}{f(x)}=\sum_{i,j\ge 0}\left(a_i t^ix^i\right )\left(D_j(f) x^j\right)=
\sum_{k\ge 0}\left(\sum\limits_{i=0}^k a_iD_{k-i}(f)\cdot t^i\right)\ .$$
 
This together with Theorem \ref{th:quotient of series} proves (a).

Prove (b) now. Indeed, 
$$f(x)y f(x)^{-1}=\sum_{i,j\ge 0}\left(a_i x^i\right )y\left(D_j(f) x^j\right)=
\sum_{k\ge 0}\left(\sum\limits_{j=0}^k a_iD_{k-i}(f)\cdot x^iyx^{k-i}\right)$$

This   proves (b).

\end{proof}

Now we can finish the proof of Theorem \ref{th:general conjugation}. Indeed, suppose that $P_{f,k}(t)$ is factored as 
$$P_{f,k}(t)=a_k(t-q_{1k})\cdots (t-q_{kk})\ .$$
Then, by Proposition \ref{pr:conjugation general}(a),  $a_iD_{k-i}=a_k(-1)^{k-i}e_{k-i}(q_{1k},\ldots,q_{kk})$ for $i=0,\ldots,k$.
Therefore, in the notation of  Proposition \ref{pr:conjugation general}(b), 
$$z_k=\sum\limits_{i=0}^k a_k(-1)^{k-i}e_{k-i}(q_{1k},\ldots,q_{kk})\cdot x^iyx^{k-i}=a_k (ad~x)^{{\bf q}_k}(y)$$
for all $k\ge 1$, which together with Proposition \ref{pr:conjugation general}(b) verifies \eqref{eq:conj general}.

Theorem \ref{th:general conjugation} is proved. \endproof

\begin{proposition} $\displaystyle{P_{e_q^t,k}(t)=\frac{(t-1)(t-q)\cdots (t-q^{k-1})}{[k]_q!}}$ for all $k\ge 1$.
\end{proposition}

\proof It suffices to show that $P_{e_q^t,k}(q^a)=0$ for all $0\le a<k$. We proceed by induction in such pairs $(a,k)$. If $a=0$, then we have nothing to prove since $P_{f,k}(1)=0$ for all $f$.

Using Theorem \ref{th:quotient of series}, we obtain:
$$P_{f,k}(q^a)=\sum\limits_{i=0}^k P_{f,i}(q^b)P_{f,k-i}(q^{a-b})q^{(a-b)i} \ .$$
Taking $f=e_q^t$,  $1\le b\le a <k$, and using the inductive hypothesis, this gives $P_{f,k}(q^a)=0$ for any $1\le a<k$. 

The proposition is proved. \endproof

\begin{corollary}

\label{pr:qbinom factorization} For all $k\ge 1$ one has: 
$\displaystyle{\det \begin{pmatrix}
1&\frac{t}{[1]_q}&\frac{t^2}{[2]_q}&\ldots &\frac{t^k}{[k]_q}\\
1&\frac{1}{[1]_q}&\frac{1}{[2]_q}&\ldots &\frac{1}{[k]_q}\\
0&1&\frac{1}{[1]_q}&\ldots &\frac{1}{[k-1]_q}\\
\ldots&\ldots&\ldots&\ldots&\ldots \\
0&0&\ldots &1&\frac{1}{[1]_q} \end{pmatrix} =\frac{(1-t)(q-t)\cdots (q^{k-1}-t)}{[k]_q!}}$.

\end{corollary}

\noindent {\bf Proof of Proposition \ref{pr:Hall-Littlewood}}. Indeed, if $f(t)$ is as in Proposition \ref{pr:Hall-Littlewood}, then 
$$\displaystyle{\frac{f(tu)}{f(u)}=\prod\limits_{k\ge 1} \frac{1-x_ktu}{1-x_ku}=\sum\limits_{k\ge 0} Q_{(k)}({\bf x};t)u^k}$$
by \cite[Equations (2.10) and (2.13)]{M}. 
This and Theorem \ref{th:quotient of series} imply that $P_{f,k}=Q_{(k)}({\bf x};t)$ for all $k\ge 0$, which proves the first assertion of 
Proposition \ref{pr:Hall-Littlewood}.

To prove the second assertion, note that $a_k=(-1)^ke_k({\bf x})$ for all $k\ge 0$ because of the well-known formula
 (see e.g., \cite[Section 1.2]{M}):
$$\prod\limits_{k\ge 1} (1-x_kt)=\sum_{k\ge 0} (-1)^ke_k({\bf x})t^k \ .$$
This and the first assertion of Proposition \ref{pr:Hall-Littlewood} imply the second assertion of the proposition.
\endproof

%
%
%
%
%
%
%

\end{document}